\documentclass[10pt]{amsart}
\usepackage{amsmath,amsfonts,amssymb,amsthm}
\usepackage[alphabetic]{amsrefs}
\usepackage{hyperref}
\usepackage{graphicx,color}
\usepackage[top=1in, left=1in, right=1in, bottom=0.85in]{geometry}

\newtheorem{theorem}{Theorem}[section]
\newtheorem{lemma}[theorem]{Lemma}
\newtheorem{prop}[theorem]{Proposition}
\newtheorem*{theoremA}{Theorem A}
\newtheorem*{theoremB}{Theorem B}
\newtheorem*{theorem*}{Theorem}
\newtheorem*{corollary}{Corollary 1}
\theoremstyle{definition} 
\newtheorem{definition}{Definition}[section]

\theoremstyle{remark}

\newcommand{\mr}{\mathrm}
\newcommand{\wt}{\widetilde}
\newcommand{\es}{\stackrel{e}{\sim}}
\newcommand{\esc}{\stackrel{\bar e}{\sim}}
\newcommand{\tw}{\widetilde w}
\newcommand{\tx}{\widetilde x}
\newcommand{\tB}{\widetilde B}
\newcommand{\tS}{\widetilde S}
\newcommand{\tX}{\widetilde X}
\newcommand{\tz}{\widetilde z}
\newcommand{\tu}{\widetilde u}
\newcommand{\ty}{\widetilde y}
\newcommand{\ain}{\alpha_{in}}
\newcommand{\aou}{\alpha_{out}}
\newcommand{\ovv}{\overline v_0}
\newcommand{\ra}{\rangle}
\newcommand{\la}{\langle}

\begin{document}

\title[Combinatorial negative curvature and three-manifolds]{Combinatorial negative curvature and triangulations of three-manifolds}

\author{Damian Osajda}
\address{Instytut Matematyczny,
Uniwersytet Wroc\l awski\\
pl.\ Grunwaldzki 2/4,
50--384 Wroc{\l}aw, Poland}
\address{Universit\"at Wien, Fakult\"at f\"ur Mathematik\\
Oskar-Morgenstern-Platz 1, 1090 Wien, Austria.
}
\email{dosaj@math.uni.wroc.pl}

\date{\today}

\begin{abstract}
We introduce and study local combinatorial conditions on a simplicial complex, implying Gromov hyperbolicity of its universal cover.
We apply the theory to Thurston's problem on $5/6^*$--triangulations of $3$--manifolds, providing a new proof and generalizing
the original result. We indicate further applications.
\end{abstract}

\maketitle

\section{Introduction}
\label{s:intro}
A \emph{$5/6^*$--triangulation} of a $3$--manifold is a triangulation in which every edge has \emph{degree} $5$ or $6$ (i.e.\ there are $5$ or $6$ tetrahedra around each edge), and every triangle contains only one edge of degree $5$.
Thurston conjectured that every closed $3$--manifold that admits a $5/6^*$--triangulation has Gromov hyperbolic
fundamental group. Elder, McCammond and Meier \cite{EMcCM} established this statement, via a computer-assisted proof.
\medskip

In this paper we place the question on hyperbolic triangulations of $3$--manifolds into the frame of a combinatorial nonpositive
curvature. We introduce a new local combinatorial condition of \emph{$m$--location} (where $m=7,8,9,\ldots$) for flag simplicial complexes. 
When $m \geqslant 8$ this condition is an analogue of the negative curvature, as the following main result of the paper 
shows.
\begin{theoremA}
Let $X$ be a simply connected $8$--located locally $5$--large simplicial complex. Then the $1$--skeleton of $X$, equipped with
the standard path metric, is Gromov hyperbolic.
\end{theoremA}
There are many $8$--located simplicial complexes and groups acting on them geometrically resulting from various notions of combinatorial
nonpositive curvature appearing in the literature (see Section~\ref{s:fin} for a discussion). 
However, with the notion of $m$--location we enter for the first time the world of 
triangulations of manifolds in dimensions above $2$.
\begin{theoremB}
Every $5/6^*$--triangulation of a $3$--manifold is an $8$--located locally $5$--large simplicial complex.
\end{theoremB}
This provides a new proof of Thurston's conjecture. Our methods are quite elementary and we do not use computer computations. Moreover, Theorem B shows that Theorem A is a generalization of the original statement. There are 
$8$--located triangulations of $3$--manifolds that are not $5/6^*$--triangulations. Furthermore, our proof provides a uniform bound
on the hyperbolicity constant and describes the structure of combinatorial balls in $5/6^*$--triangulations.
\medskip

Clearly, Theorem A applies in a much broader context than the one of $3$--manifolds. It concerns general simplicial complexes. In Section~\ref{s:fin} we indicate other applications.
We also relate $8$--location to the well known notion of ``combinatorial nonpositive curvature" --- \emph{systolicity}.
Note that systolic complexes are in a sense very (asymptotically) far from triangulations of manifolds above dimension $2$. 
This makes the $8$--location very interesting, since it may be applied, unlike systolicity, to some classical spaces and groups.
In the same Section~\ref{s:fin} we briefly discuss $m$--location, for $m<8$, as an analogue of a nonpositive, but not negative,
curvature.

The proof of Theorem A is presented in Section~\ref{s:hyp}. It uses a local-to-global technique developed by the author
in \cite{O-cnpc}. The main global property of simply connected $8$--located complexes --- the property SD' --- is a variation of the main feature
of \emph{weakly systolic complexes} --- the \emph{simple descent property}.
In Section~\ref{s:568} we prove Theorem B. 

\medskip

\noindent
{\bf Acknowledgment.} I thank Tadeusz Januszkiewicz and Misha Kapovich for discussions on triangulations of manifolds.

This research was supported by Narodowe Centrum Nauki, decision no DEC-2012/06/A/ST1/00259, and by the ERC grant ANALYTIC no.\ 259527.
\section{Preliminaries}
\label{s:prel}


\subsection{Simplicial complexes}
\label{s:simp}

Let $X$ be a simplicial complex. The $i$--skeleton of $X$ is denoted by $X^{(i)}$.
A subcomplex $Y$ of $X$ is \emph{full} if
every subset $A$ of vertices of $Y$ contained in a simplex of $X$, is
contained in a simplex of $Y$.
For a finite set $A=\{ v_1,\ldots,v_k \}$ of vertices of $X$, by $\la A \ra$ or by $\langle v_1,\ldots,v_k \rangle$ we denote the \emph{span} of $A$, i.e.\ the smallest full subcomplex of $X$ containing $A$.
Thus ``$\langle A \rangle \in X$" or ``$\langle v_1,v_2,\ldots \rangle \in X$" mean that the corresponding sets span a simplex in $X$.
We write $v\sim v'$ (respectively, $v\nsim v'$) if $\langle v,v' \rangle \in X$ (respectively, $\langle v,v' \rangle \notin X$). Moreover, we write $v\sim v_1,v_2,\ldots$ (respectively, $v\nsim v_1,v_2,\ldots$) when $v\sim v_i$ (respectively, $v\nsim  v_i$) for $i=1,2,\ldots$.
A simplicial complex $X$ is \emph{flag} whenever every finite set of vertices of $X$ joined pairwise by edges in $X$, is contained in a simplex of $X$.
A \emph{link} of a simplex $\sigma$ of $X$ is a simplicial complex $X_{\sigma}=\{ \tau | \; \tau \in X \; \& \; \tau \cap \sigma=\emptyset \; \& \; \la \tau \cup \sigma\ra\in X \}$.

Let $k\geqslant 4$. A \emph{$k$--cycle} $(v_1,\ldots,v_{k})$ is a triangulation of a circle consisting of $k$ vertices: $v_1,\ldots,v_{k}$, and $k$ edges: $\langle v_i,v_{i+1}\rangle $ and $\langle v_k,v_{1}\rangle $.
A \emph{$k$--wheel (in $X$)} $(v_0;v_1,\ldots,v_k)$ (where $v_i$'s are vertices of
$X$) is a subcomplex of $X$ such that $(v_1,\ldots,v_k)$ is a full cycle, and $v_0 \sim v_1,\ldots,v_k$.
A flag simplicial complex $X$ is \emph{$k$--large} if there are no full $j$--cycles in $X$, for $j<k$.
$X$ is \emph{locally $k$--large} if all its links are $k$--large. Observe that local $5$--largeness means that there are no $4$--wheels.

If not stated otherwise, speaking about a simplicial complex $X$,
we always consider the metric on the $0$--skeleton $X^{(0)}$,
defined as the number of edges in the shortest $1$--skeleton path joining two given vertices.
We denote this metric by $d(\cdot,\cdot)$.
Given a nonnegative integer $i$ and a vertex $v\in X$, a (combinatorial) \emph{ball} $B_i(v,X)$
(respectively, \emph{sphere} $S_i(v,X)$) \emph{of radius $i$
around $v$} (or \emph{$i$--ball} around $v$) is a full subcomplex of $X$ spanned by vertices at distance
at most $i$ (respectively, at distance $i$) from $v$. We use the notation $B_i(v)$ and $S_i(v)$ for, respectively,
$B_i(v,X)$ and $S_i(v,X)$, when it is clear what is the underlying complex $X$. 


\subsection{Location}
\label{s:loc}
\begin{figure}[h!]
\centering
\includegraphics[width=0.7\textwidth]{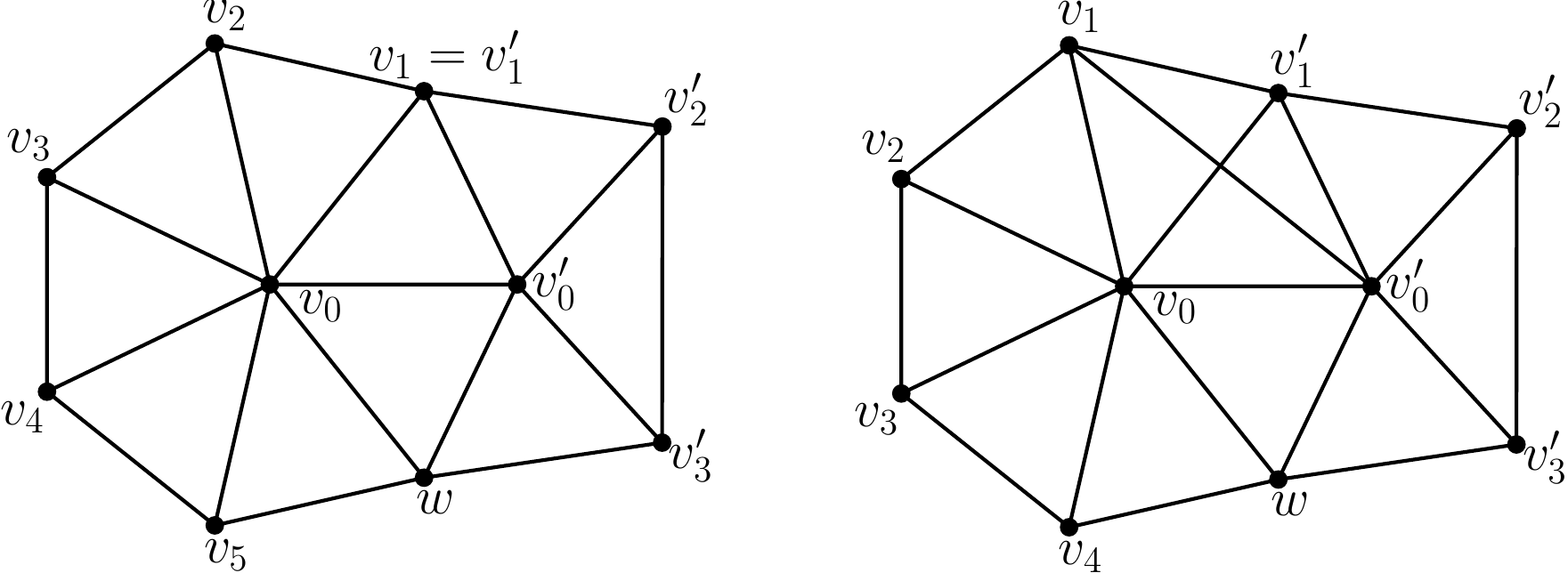}
\caption{Two types of dwheels with the boundary length $8$: a $(7,5)$--dwheel (left), and a $(6,5)$--dwheel (right).}
\label{f:dwheels}
\end{figure}

A \emph{$(k,l)$--dwheel} (from ``double wheel") $dW=(v_0,v_0',w;v_1,\ldots, v_{k-2}; v_1',\ldots,v_{l-2}')$ is the union of a $k$--wheel 
$(v_0;v_1,\ldots, v_{k-2},w,v_0')$ and an $l$--wheel $(v_0';v_1',\ldots, v_{l-2}',w,v_0)$, with $v_1=v_1'$ or $v_1\sim v_1'$ --- see Figure~\ref{f:dwheels}.
The \emph{boundary} of the dwheel $dW$ is the cycle $(v_1,\ldots,v_{k-2},w,v_{l-2}',\ldots,v_1')$, and the \emph{boundary length}
 is $k+l-4$ if $v_1=v_1'$ or $k+l-3$ otherwise.

\begin{definition}[$m$--location]
\label{d:loc}
A flag simplicial complex is \emph{$m$--located} if every dwheel with the boundary length at most $m$ is contained in a
$1$--ball --- see Figure~\ref{f:locSD}.
\end{definition}

\begin{figure}[h!]
\centering
\includegraphics[width=0.6\textwidth]{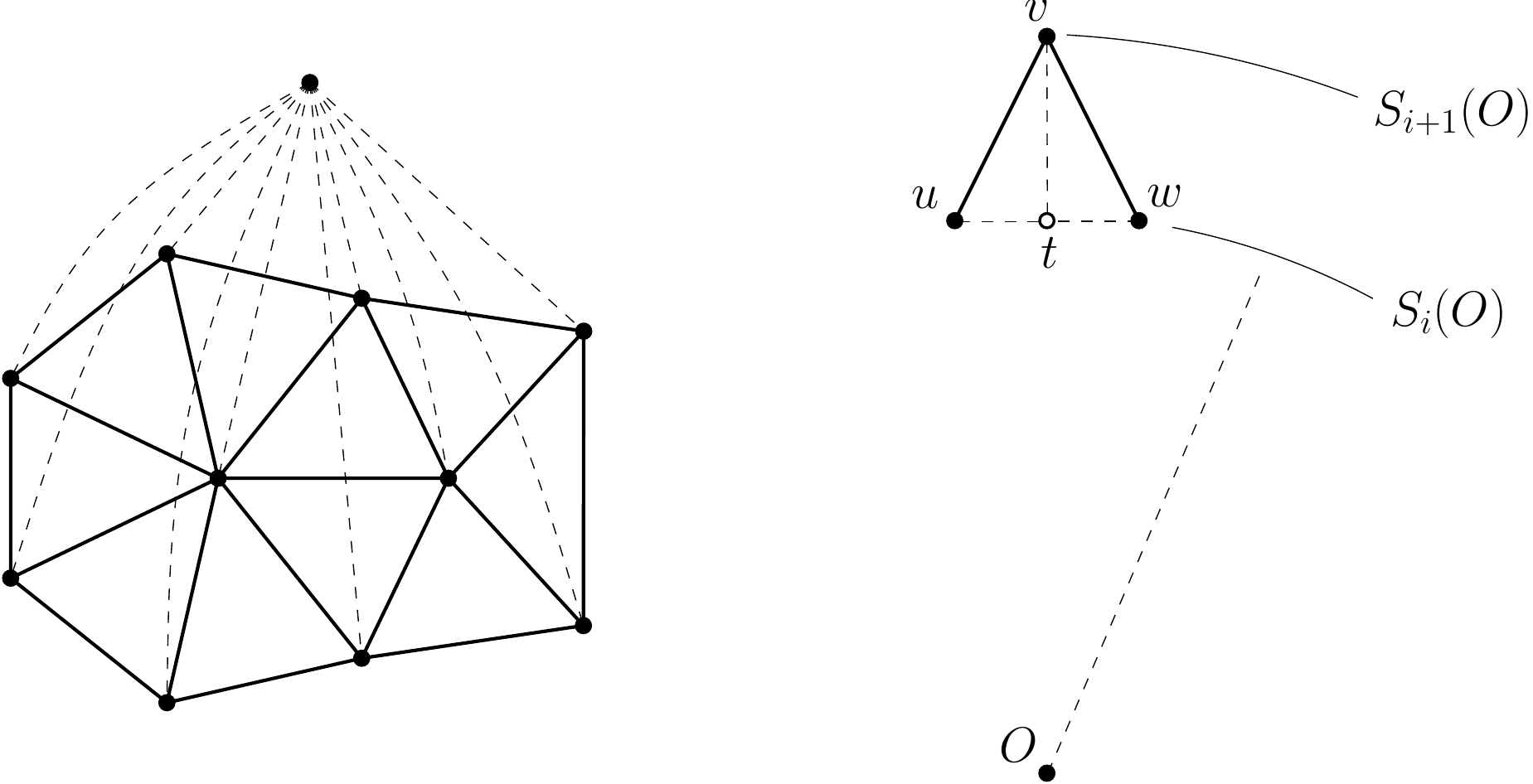}
\caption{$8$--location (left), and the vertex condition (V) for the property SD' (right).}
\label{f:locSD}
\end{figure}

In this article by a \emph{covering} we mean a \emph{simplicial covering}, that is a simplicial map restricting to
isomorphisms from $1$--balls onto spans of their images. In particular, it follows that wheels are mapped isomorphically onto wheels, and thus the
following result holds.

\begin{prop}
\label{p:cov}
A covering of a locally $k$--large complex is itself locally $k$--large.
A covering of an $m$--located complex is itself $m$--located.
\end{prop}



\section{Hyperbolicity}
\label{s:hyp}
The goal of this section is to prove Theorem A. In Subsection~\ref{s:glo} we introduce a global combinatorial property $SD'$ that,
together with $8$--location and local $5$--largeness, implies hyperbolicity --- Theorem~\ref{t:hyp}. In the subsequent 
Subsection~\ref{s:lotoglo}, we prove a local-to-global Theorem~\ref{t:logl} implying directly Theorem A.

\subsection{Global condition}
\label{s:glo}

\begin{definition}[Property ${SD'}$]
\label{d:sd}
Let $X$ be a flag simplicial complex, let $O$ be its vertex, and let $n\in \{0,1,2,3,\ldots\}$.
We say that
$X$ \emph {satisfies the property ${SD'}_n(O)$} if for every $i=1,2,\ldots,n$ the following two conditions hold.
\begin{description}
  \item[({\bf T}) \rm(Triangle condition)] For every
edge $e \in S_{i+1}(O)$ the intersection
 $X_e\cap B_i(O)$ is non-empty.
  \item[({\bf V}) \rm{(Vertex condition)}] For every vertex $v \in S_{i+1}(O)$, and for every two vertices
$u,w\in X_v\cap B_i(O)$, there exists a vertex $t\in X_v\cap B_i(O)$ such that $t\sim u,w$ --- see Figure~\ref{f:locSD}.
\end{description}

$X$ satisfies the property \emph{$SD'(O)$} (respectively, the  property \emph{$SD'$}) if $SD'_n(O)$ holds for every $n$
(respectively, for every $n$ and $O$). 
\end{definition}

\begin{prop}[Simple connectedness]
\label{p:sconn}
Let $X$ be a simplicial complex satisfying the property $SD'(O)$, for some vertex $O$. Then $X$ is simply connected.
\end{prop}
\begin{proof}
The proof follows closely the proof of \cite[Lemma 5.5]{BCCGO}. Here, instead of the quadrangle condition one
uses the vertex condition (V) from Definition~\ref{d:loc}.
\end{proof}

\begin{lemma}
\label{l:21}
Let $X$ be an $8$--located locally $5$--large simplicial complex satisfying the property $SD_n'(O)$, for some vertex $O$. 
Let $v\in S_{n+1}(O)$ and let $y,z\in X_v\cap B_{n}(O)$ be such that $y\nsim z$.
Let $x\in X_v\cap B_{n}(O)$ be a vertex adjacent to $y,z$, given by the vertex condition (V).
Let $y',z'\in B_{n-1}(O)$ be vertices with $\la x,y,y' \ra,\la x,z,z' \ra \in X$, given by the triangle condition (T).
Then $y'\neq z'$; $y'\nsim z$; $y\nsim z'$; and $y'\sim z'$. 
\end{lemma}
\begin{proof}
By $5$--largeness we have immediately that $y'\neq z'$; $y'\nsim z$; and $y\nsim z'$. For the rest we proceed by contradiction.
Assume that $y'\nsim z'$. Then, by the vertex condition (V), there is a vertex $x'\in S_{n-1}(O)$
adjacent to $x,y,z$. Note that, by $5$--largeness,  it is not possible that $y\sim x'\sim z$, thus there is a $6$--wheel $(x;v,z,z',x',y',y)$, or 
a $5$--wheel $(x;v,z,z',x',y)$, or $(x;v,z,x',y',y)$.
By the triangle (T) and the vertex (V) condition there are vertices $y'',x'',z'' \in S_{n-2}(O)$ with $\la x',y',y'' \ra,
\la x',z',z'' \ra \in X$, and $x''\sim x',y'',z''$. 
Possibly $x''=z''$, but $x''\neq y''$, by $5$--largeness.
In any case we obtain an
$(k,l)$--dwheel $dW$, spanned by vertices in $\{v,x,y,z,x',y',z',x'',y'',z''\}$, with the boundary length at most $8$. Thus, by $8$--location,   
$dW$ is contained in a $1$--ball, which contradicts the fact that $d(v,y'')=3$.
\end{proof}

\begin{theorem}[Hyperbolicity]
\label{t:hyp}
Let $X$ be an $8$--located locally $5$--large simplicial complex satisfying the property $SD'$. Then $X^{(0)}$ equipped
with a path metric induced from $X^{(1)}$ is $\delta$--hyperbolic, for a 
universal constant $\delta$.
\end{theorem}
\begin{proof}
We use a criterion by Papasoglu \cite{Papa}, i.e.\ we reduce the proof to showing that intervals are uniformly thin.
Let $O,O'$ be two vertices and let $I$ be the interval between them, i.e.\ the set of vertices
lying on geodesics between $O$ and $O'$. By $I_k$ we denote the intersection $S_{k}(O)\cap I=S_k(O)\cap S_{n-k}(O')$, where 
$n=d(O,O')$. We show that for every $k\leqslant n$, for every two vertices $v,w \in I_k$,
we have $d(v,w)\leqslant 2$. This shows also that the hyperbolicity constant is universal.
\medskip

By contradiction --- let $k$ be the maximal number such that there are vertices $v,w \in I_k$ with $d(v,w)>2$.
Then in $I_{k+1}$ there exist vertices $v'\sim v$, and $w'\sim w$ with $d(v',w')\leqslant 2$.
By the vertex condition (V), there is a vertex $z\sim v',w'$ in $I_{k+1}$, possibly with $z=w'$.
By the triangle condition (T) there are vertices $v'',w'' \in I_k(O)$, with $v''=w''$ if $z=w'$, such that 
$v''\sim v',z$, and $w''\sim z,w'$. 
By the vertex condition (V), in $I_k$ there exist vertices $s\sim v,v',v''$; $t\sim v'',z,w''$; $u \sim w,w',w''$, possibly,
with $s=v''$, $t=w''$, $u=w$. Among the vertices $t,w'',u,w$ we choose the first one (in the given order), that is at distance
$3$ from $v$. Denote this vertex by $v'''$. Using the vertex condition (V), we obtain a full path $v_1v_2v_3v_4$ in $I_k$ of diameter $3$, with
$v_1=v$ and $v_4=v'''$. 
In the remaining part of the proof we show that the existence of such path is impossible. This leads to a contradiction yielding  
the theorem.

For $i=1,2,3$, let $w_i\sim v_i,v_{i+1}$ be a vertex in $I_{k-1}$ given by the triangle condition (T). 
All the possible cases (up to renaming vertices) of mutual relations between vertices $w_i$ are shown in Figure~\ref{f:hyp} --- Cases:
I, II, and III. 
\medskip

\noindent
{\bf Case I.} In this case we assume that for all triples of vertices $w_i$ as above we are not in Case II or III, that is $w_i\neq w_j$ , for $i\neq j$. First, suppose that $w_1\sim w_2 \sim w_3$. We may assume that $w_1\nsim w_3$. By (T), there are vertices 
$u_1,u_2 \in I_{k-2}$ with $u_j\sim w_j,w_{j+1}$. 
If $u_1=u_2$ then we are in Case I(a) (see Figure~\ref{f:hyp}). Further we assume that there are no such $u_i$. 
If $u_1\sim u_2$ then, by (T), there is a vertex $t\in I_{k-3}$ 
adjacent to $u_1,u_2$ --- Case I(b) (see Figure~\ref{f:hyp}). If not then, by (V), there is a vertex $u'\in I_{k-2}$ adjacent to $u_1,u_2$.
If $u'\sim w_1$ then we are in the previous case taking $u'$ instead of $u_1$. Similarly for $u'\sim w_3$, thus further we assume
that $u'\nsim w_1,w_3$.
By (T), in $I_{k-3}$ there exist $t_1,t_2$ with $t_i \sim u',u_i$. By Lemma~\ref{l:21}, we have $t_1\neq t_2$ and $t_1\sim t_2$.
It follows that there is a $(7,5)$--dwheel $(w_2,u',u_2;u_1,w_1,v_2,v_3,w_3;u_1,t_1,t_2)$ $dW$. By $8$--location, $dW$ is
contained in a $1$--ball --- contradiction, since $d(v_2,t_1)= 3$.

\begin{figure}[h!]
\centering
\includegraphics[width=0.7\textwidth]{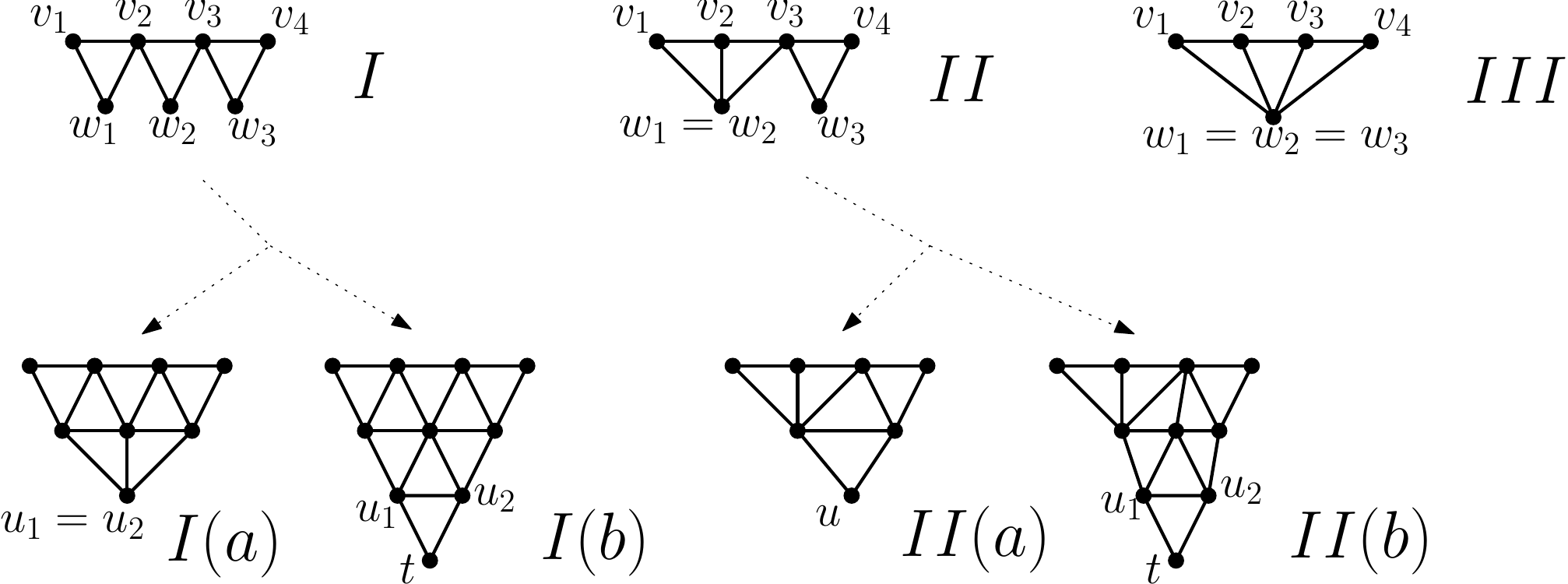}
\caption{Possible cases in the proof of Theorem~\ref{t:hyp}.}
\label{f:hyp}
\end{figure}

 Now we suppose that $w_1 \sim w_2$, and $w_2\nsim w_3$. Then, by (V), in $I_{k-1}$ there is $w'\sim v_3,w_2,w_3$. We may assume that $w'\nsim w_1$. By (T), in 
$I_{k-2}$
there are $u_2,u_3$, with $u_i\sim w',w_i$. By Lemma~\ref{l:21},  $u_2 \neq u_3$; $u_2\nsim w_3$; $w_2\nsim u_3$; and $u_2\sim u_3$. Therefore, $w_1,v_2,v_3,w_3,u_3$ are contained in a
 $(5,m)$--dwheel $dW$, with the boundary length at most $8$. By $8$--location, $dW$ is
contained in a $1$--ball $B_1(w'')$. Replacing $w_2$ by $w''$ we are in Case I(a) or I(b).

If $w_1 \nsim w_2 \nsim w_3$ then, by (V), in $I_{k-1}$ there are $w_1',w_2'$ with $w_j'\sim v_{j+1},w_j,w_{j+1}$.
By (T), in $I_{k-2}$ there exist $u_1,u_1',u_2,u_2'$ with $u_j\sim w_j,w_j'$, and $u_j'\sim w_j',w_{j+1}$.
By Lemma~\ref{l:21} $u_1\neq u_1'$; $u_1 \nsim w_2$; $w_1 \nsim u_1'$; and $u_1\sim u_1'$. Thus the vertices $u_1,w_1,v_2,v_3,w_2',u_2$ are contained in a $(5,m)$--dwheel $dW$, with the boundary length at most $8$.
By $8$--location, $dW$ is
contained in $B_1(w)$, and hence, replacing $w_2$ by $w$, we are in the previous case.  
\medskip

\noindent
{\bf Case II.} Proceeding similarly as in Case I above (i.e.\ using the conditions (V) and (T), and Lemma~\ref{l:21}), we 
come to the two possible cases: II(a) and II(b), as in Figure~\ref{f:hyp} --- with $u,u_1,u_2\in I_{k-2}$ and $t\in I_{k-3}$.
\medskip

In the remaining part of the proof we analyze separately all the cases: I(a), I(b), II(a), II(b), and III, showing that they lead to contradiction.
To do this we ``project" edges and vertices towards $O'$ now: By (T), in $I_{k+1}$ there exist vertices $p_1,p_2,p_3$ (possibly some of them coinciding), with $p_j \sim v_j,v_{j+1}$. By (V), we have that 
$d(p_j,p_{j+1})\leqslant 2$.
\medskip

\noindent
\emph{Ad Case III: }This leads to a contradiction, since $d(v_1,v_4)=3$.
\medskip

\noindent
\emph{Ad Case II(a): }By Lemma~\ref{l:21}, we have $p_1\nsim v_3$; $v_1\nsim p_2$; $p_1\neq p_2$; and $p_1\sim p_2$. Therefore, the vertices $p_1,v_1,w_1,w_3, v_4,p_3$ are then contained in a $(5,m)$--dwheel  $dW$ with the boundary length at most $8$. 
Thus, by $8$--location,
$dW$ is contained in a $1$--ball --- contradiction, since $d(v_1,v_4)=3$.
\medskip

\noindent
\emph{Ad Case II(b): }By Lemma~\ref{l:21}, we have $p_1\nsim v_3$; $v_1\nsim p_2$; $p_1\neq p_2$; and $p_1\sim p_2$. 
By (V), there is $p'\in I_{k+1}$ with $p'\sim p_2,p_3,v_3$, possibly with $p'=p_3$. We may assume that $p'\nsim p_1$.
By (T), in $I_{k+2}$ there are vertices $q_1\sim p_1,p_2$ and $q_2\sim p_2,p'$, with $d(q_1,q_2)\leqslant 2$, by (V).
Therefore the vertices $q_1,p_1,v_1,w_1,v_3,p',q_2$ are contained in 
a $(5,m)$--dwheel  $dW$ with the boundary length at most $8$. Thus, by $8$--location,
$dW$ is contained in a $1$--ball --- contradiction, since $d(w_1,q_1)=3$.
\medskip

\noindent
\emph{Ad Case I(a): }The vertices $p_2,v_2,w_1,u_1,w_3,v_4,p_3$ are contained in a $(5,m)$--dwheel  $dW$ with the boundary length at most 
$8$. Thus, by $8$--location,
$dW$ is contained in a $1$--ball --- contradiction, since $d(u_1,p_3)=3$.
\medskip

\noindent
\emph{Ad Case I(b): }If $p_1=p_2$ or $p_1\sim p_2$ then we obtain a $(6,m)$--dwheel containing $p_1,v_1,w_1,u_1,u_2,w_3,v_3,p_2$,
 leading to contradiction, by $8$--location. Similarly, when $p_2=p_3$ or $p_2\sim p_3$. Thus, further we assume that
$d(p_j,p_{j+1})=2$. By (V), in $I_{k+1}$ there is a vertex $p'\sim p_1,p_2$. We may assume  that $v_1\nsim p' \nsim v_3$.
By (T), in $I_{k+2}$ there exist vertices $q_1,q_2$, such that $q_j\sim p_j,p'$. 
By Lemma~\ref{l:21}, $q_1 \neq q_2$; $q_1\nsim p_2$; $p_1\nsim q_2$; and $q_1\sim q_2$.
Thus there is a $(7,5)$--dwheel $(v_2,p',p_2;p_1,v_1,w_1,w_2,v_3;p_1,q_1,q_2)$, which lies in a $1$--ball,
by $8$--location. This however contradicts the fact that $d(w_1,q_1)=3$.
\end{proof}

\subsection{Local-to-global}
\label{s:lotoglo}

In this section we prove the following result implying Theorem A from Introduction.

\begin{theorem}[Local-to-global]
\label{t:logl}
Let $X$ be an $8$--located locally $5$--large simplicial complex. Then its universal cover $\wt X$ is an $8$--located locally $5$--large simplicial complex satisfying the property $SD'$. In particular, $\wt X^{(1)}$ is $\delta$--hyperbolic.
\end{theorem}
\begin{proof}
The proof follows closely --- up to much of the notations --- the proof of the analogous Theorem 4.5 from \cite{O-cnpc} (compare
also the proof of \cite[Theorem 1]{BCCGO}).
We construct the universal cover $\widetilde X$ of $X$ as an increasing
union $\bigcup_{i=1}^{\infty}\widetilde B_i$ of combinatorial balls. The covering map is then the union
$$
\bigcup_{i=1}^{\infty}f_i\colon \bigcup_{i=1}^{\infty}\widetilde B_i \to X,
$$
where
$f_i \colon \widetilde B_i \to X$ is locally injective and $f_i |_{\widetilde B_j}=f_j$, for $j\leqslant i$.

We proceed by induction. Choose a vertex $O$ of $X$. Define $\wt B_0=\{ O \}$, $\widetilde B_1=B_1(O,X)$ and $f_1=\mr {Id}_{B_1(O)}$.
Assume that we have constructed the balls $\widetilde B_1,\widetilde B_2,\ldots,\widetilde B_i$ and
the corresponding maps $f_1,f_2,\ldots,f_i$ to $X$ so that the following conditions are satisfied:
\medskip

\noindent
\begin{description}
  \item[($P_i$)] $\wt B_j=B_j(O,\wt B_i)$ for $j=1,2,\ldots,i$;
  \item[($Q_i$)] $\wt B_i$ satisfies the property $SD'_{i-1}(O)$;
  \item[($R_i$)] $f_i|_{B_1(\wt w,\wt B_i)}\colon B_1(\wt w,\wt B_i) \to B_1(f_i(\wt w),X)$ is
an isomorphism onto the span of the image for $\wt w\in \wt B_{i}$ and it is
an isomorphism for $\wt w\in \wt {B}_{i-1}$.
\end{description}

Observe that those conditions are satisfied for $\wt B_1$ and $f_1$, i.e.\ that conditions ($P_1$), ($Q_1$) and ($R_1$) hold.
Now we construct $\widetilde{B}_{i+1}$ and the map
$f_{i+1}\colon \widetilde{B}_{i+1}\to X$. For a simplex $\wt {\sigma}$ of
$\wt B_i$, we denote by $\sigma$ its image $f_i(\wt {\sigma})$ in $X$.
Let $\wt S_i=S_i(O,\wt B_i)$ and let
\begin{align*}
Z=\{ (\widetilde w,z)\in \widetilde S_i^{(0)}\times X^{(0)}|\; \; z\in X_{w}\setminus
f_i((\widetilde B_i)_{\widetilde{w}}) \}.
\end{align*}
Define a relation $\stackrel{e}{\sim}$ on $Z$ as follows:
\begin{align*}
(\widetilde{w},z)\stackrel{e}{\sim} (\widetilde w',z') \; \; \mr{iff} \; \; (z=z' \;\; and \;\; \langle \widetilde{w},\widetilde w'\rangle \in \wt B_i^{(1)}).
\end{align*}
The transitive closure $\esc$ of the relation $\es$ will be further used to define $\tB_{i+1}$. The following lemma
shows that $\esc$ is not ``too far" from $\es$.

\begin{lemma}
\label{l:equiv}
If $(\wt{w}_1,z)\es (\wt{w}_2,z) \es (\wt{w}_3,z)\es (\wt{w}_4,z)$ then there is $(\tx,z) \in Z$ with  
$(\tw_1,z) \es (\tx,z) \es (\tw_4,z)$.
\end{lemma}
\begin{proof}
If $\tw_j = \tw_k$, for $j\neq k$, or if $\la \tw_1,\tw_3 \ra \in \tB_{i}$, or 
$\la \tw_1,\tw_4 \ra \in \tB_{i}$, or $\la \tw_2,\tw_4 \ra \in \tB_{i}$, then the assertion trivially holds. 
Thus further we assume this is not the case.
By ($P_i$) and ($Q_i$), in $\tB_{i-1}$ there are vertices $\tu_1,\tu_2,\tu_3$ with $\tu_j \sim \tw_j, \tw_{j+1}$.
\medskip

\noindent
{\bf Claim. } For $j\neq k$, we have $\tu_j\neq \tu_k$. Furthermore, for $j=1,2$, $\tu_j \nsim \tw_{j+2}$ and
$\tw_j \nsim \tu_{j+2}$.

\medskip

\noindent
\emph{Proof of Claim. }
We show only that $\tu_1 \nsim \tw_3$. Other assertions follow in the same way.
By contradiction --- assume that $\tu_1 \sim \tw_3$. Observe that $z\nsim u_1$, by ($R_i$) and by the definition of the set $Z$, and that 
$w_1 \nsim w_3$, by ($R_i$) and our initial assumptions.
Therefore, by ($R_i$) and by our definition of $Z$ there is a $4$--wheel $(w_2;w_1,z,w_3,u_1)$ in $X$ --- contradiction 
with $5$--largeness. This finishes the proof of Claim.
\hfill $\square$
\medskip

By ($Q_i$), in $\tB_{i-1}$ there are vertices $\tu_1',\tu_2'$ with $\tu_j'\sim \tu_j,\tw_{j+1},\tu_{j+1}$, and $\tu_j'\neq \tu_j$, by Claim. 
When $\tu_j \sim \tu_{j+1}$ then $\tu_j'=\tu_{j+1}$. We may assume that $\tu_1'\nsim \tw_1$ --- if not
then we take $\tu_1'$ instead of $\tu_1$. Further, we may assume that if $\tu_1' \neq \tu_2$ then $\tu_1' \nsim \tw_3$ ---
if not than we choose $\tu_1'$ instead of $\tu_2$. By the definition of the set $Z$ and by ($R_i$), it follows that in $X$
there is a wheel spanned by vertices $w_2,z,w_3,u_2,u_1',u_1,w_1$. Furthermore, if $\tu_2' \neq \tu_2$ then we may assume that 
$\tu_2' \nsim \tw_4$ --- otherwise we replace $\tu_3$ by $\tu_2'$. It may however happen that $\tu_2'\sim \tw_2$. Then, by the definition
of the set $Z$ and by ($R_i$), in $X$ there is a $5$--wheel $(w_3;u_2',w_2,z,w_4,u_3)$. If  $\tu_2'\nsim \tw_2$ then we have 
a wheel spanned by vertices $w_3,u_2',u_2,w_2,z,w_4,u_3$. Since $u_2\sim u_2'$, in any case we obtain in $X$ a dwheel
$dW$ with the boundary length at most $8$. Therefore, by $8$--location, there is a vertex $y$ with  $dW \subseteq B_1(y,X)$. By ($R_i$) applied to the vertex $\tu_2$ there is a vertex $\ty \sim \tu_2$, with $f_i(\ty) =y$.
Again by ($R_i$), we have that all the vertices $\tw_1,\tw_2,\tw_3,\tw_4,\tu_1,\tu_2,\tu_3,\tu_1',\tu_2'$ are adjacent to
$\ty$ in $\tB_i$. Hence, to prove the lemma it is enough to show that $(\ty,z)\in Z$.
If not then, by ($R_i$), there is $\tz \in \tB_i$ such that $\la \tz,\ty \ra \in \tB_i$. By ($R_i$), we have then that $\la \tz,\tw_1 \ra \in \tB_i$, which is a contradiction, since $(\tw_1,z)\in Z$. This proves the lemma.
\end{proof}

\medskip

Observe that, by Lemma~\ref{l:equiv}, if $(\tu,z)\esc (\tw,z)$ then there is a vertex $\ty \in \tS_i$ with $(\ty,z)\in Z$ and
$\la \ty,\tu \ra, \la \ty,\tw \ra \in \tB_i$.

We define the flag simplicial complex $\wt{B}_{i+1}$ in the following way.
Its $0$--skeleton is by definition the set
$\wt{B}_{i+1}^{(0)}=B_{i}^{(0)}\cup (Z/\esc)$.
Now we define the $1$--skeleton $\wt{B}_{i+1}^{(1)}$ of $\wt{B}_{i+1}$ as follows.
Edges between vertices of $\wt B_i$ are the same as in $\wt B_i$.
Moreover, for every $\wt w\in \wt{S_i}^{(0)}$, there are edges joining $\wt {w}$ with $[\wt {w},z] \in Z/\esc$ (here
$[\wt {w},z]$ denotes the equivalence class of $(\wt {w},z)\in Z$) and
the edges joining $[\wt {w},z]$ with $[\wt {w},z']$, for $\langle z,z'\rangle\in X$.
Having defined $\wt{B}_{i+1}^{(1)}$ the higher dimensional skeleta are
determined by the flagness property.
\medskip

Definition of the map $f_{i+1}\colon \wt{B}_{i+1}^{(0)}\to X$ is clear:
$f_{i+1}|_{\wt B_i}=f_i$ and $f_{i+1}([\wt {w},z])=z$.
We show that it can be simplicially extended.
It is enough to do it for simplices in $\wt{B}_{i+1}\setminus \wt {B}_{i-1}$.
Let
$\wt {\sigma}=\langle[\wt w_{1},z_1],\ldots,[\wt w_{l},z_l],\wt w_1',\ldots,$ $\wt w_m'\rangle\in \wt{B}_{i+1}$ be a simplex.
Then, by definition of edges in $\wt{B}_{i+1}$, we have that
$\langle z_p,z_q\rangle\in X$ and $\langle z_r,w_s'\rangle\in X$, for $p,q,r\in \{ 1,2,\ldots,l\}$
and $s\in \{ 1,2,\ldots,m \}$.
Since $f_{i+1}([\wt w_p,z_p])=z_p$, $f_{i+1}(\wt w_s')=w_s'$ and
since $f_i$ was simplicial, it follows that
$\langle \{ f_{i+1}(\wt {w})|\;\; \wt {w}\in \wt {\sigma}\} \rangle \in X$. Hence, by the simplicial extension, we can define the map $f_{i+1}\colon \wt{B}_{i+1}\to X$.
\medskip

Now we check that $\wt{B}_{i+1}$ and $f_{i+1}$ satisfy conditions ($P_{i+1}$), ($Q_{i+1}$) and ($R_{i+1}$).
The proof of the conditions ($P_{i+1}$) and ($R_{i+1}$) is exactly the same as the one in the proof of Theorem 4.5 in \cite{O-cnpc} (even
up to the notations; see
page 12 there).
\medskip

\noindent
{\bf Condition} ($Q_{i+1}$).
By the condition ($Q_i$) it is enough to verify the triangle condition (T) and the vertex condition (V) from Definition~\ref{d:sd}
only for, respectively, edges and vertices in $\tS_{i+1}$. 
By the definition of edges in ${\tS_{i+1}}$ it is clear that
$(\tB_{i+1})_{e}\cap \wt{B_i}$ is non-empty, for an edge $e\in \tS_{i+1}$. Therefore, the condition (T) is satisfied.
The vertex condition (V), for a vertex $[\tw,z]\in \tS_{i+1}$ follows immediately from the definition of edges in $\tB_{i+1}$,
and from Lemma~\ref{l:equiv}.

\medskip

Having established conditions ($P_{i+1}$), ($Q_{i+1}$) and ($R_{i+1}$) we are able to construct a complex
$\wt X=\bigcup_{i=1}^{\infty} \wt B_i$ and a map
$f=\bigcup_{i=1}^{\infty} f_i\colon \wt X \to X$ with the following properties.
The complex $\wt X$ satisfies the property $SD'_n(O)$ for every $n$ and the map $f$ is a covering map. 
Thus, by Proposition~\ref{p:cov}, the cover $\tX$ is $8$--located and locally $5$--large.
From Proposition~\ref{p:sconn} it follows that $\tX$ is simply connected and therefore $\tX$ is the universal cover of $X$.
Since the vertex $O$ was chosen arbitrarily in our construction and since the universal cover of $X$ is unique it follows that 
$\tX$ satisfies the property $SD'_n(O)$ for every vertex $O$ and for every natural number $n$.
This finishes the proof of the theorem.
\end{proof}

\section{$5/6^*$--triangulations are $8$--located}
\label{s:568}
The goal of this section is to show that $5/6^*$--triangulations are $8$--located locally $5$--large simplicial complexes, i.e.\ 
to prove Theorem B from Introduction. 

\subsection{$5/6^*$--triangulations of a $2$--sphere.}
\label{s:56}
A link of a vertex in a $5/6^*$--triangula\-tion of a $3$-manifold is a \emph{$5/6^*$--triangula\-tion of a $2$--sphere}, that is, 
each vertex has degree (the number of edges adjacent to) $5$ or $6$ and no two
vertices of degree $5$ are connected by an edge. A \emph{soccer tiling of a $2$--sphere} is the cellulation dual to a $5/6^*$--triangulation of 
a $2$--sphere --- see \cite{EMcCM}. Its cells --- pentagons and hexagons --- correspond to vertices of the triangulation, its vertices 
correspond to triangles. A \emph{soccer diagram} is a subcomplex of a soccer tiling, which is homeomorphic to a
disc. A path in the boundary of a soccer diagram which is contained in a single cell is called an \emph{exposed path}.

\subsection{$8$--location}
\label{s:s568}
Throughout this subsection we assume that $X$ is a $5/6^*$--triangula\-tion of a $3$--manifold.

\begin{lemma}
\label{l:wheels}
$X$ is locally $5$--large and, for $k=5,6$, every $k$--wheel $W$ in $X$ is contained in a link of a vertex $v\notin W$. 
\end{lemma}

The lemma above follows immediately from Lemma~\ref{l:cycles} below. The latter is an elementary result concerning $5/6^*$--triangulations 
of a $2$--sphere. For brevity, in the proof we use results from \cite[Section 3]{EMcCM}.
\begin{lemma}
\label{l:cycles}
Let $Y$ be a $5/6^*$--triangulation of a $2$--sphere. There are no full $4$--cycles in $Y$, and, for $k=5,6$, 
every full $k$--cycle in $Y$ is contained in a wheel. 
\end{lemma}
\begin{proof}
Let $k\in \{ 4,5,6 \}$ and assume there is a full  $k$--cycle $c=(c_1,c_2,\ldots,c_k)$ in $Y$. In the soccer tiling $Y'$ corresponding to $Y$,
we have then the chain of cells $c_1,c_2,\ldots,c_k$, with $c_i$ and $c_{i+1}$ sharing an edge (here and further we use the convention
$c_{k+j}=c_j$). By $c$ we denote the subcomplex of $Y'$ being the union of $c_1,c_2,\ldots,c_k$. 
Observe that if, for $k-1\geqslant j\geqslant 2$, $c_i$ and $c_{i+j}$ have an edge in common then $c$ (in $Y$) is not full, thus further we assume 
this is not the case. Consider now the two \emph{boundary paths} $\ain$ and $\aou$ of $c$ in $Y'$ --- see Figure~\ref{f:6cycle}.
Since $k\leqslant 6$ and cells are pentagons or hexagons, the sum of lengths $|\ain |+|\aou|$ is at most $6\cdot 4 =24$.
Therefore, without loss of generality we assume that $|\ain|\leqslant 12$.
\begin{figure}[h!]
\centering
\includegraphics[width=0.7\textwidth]{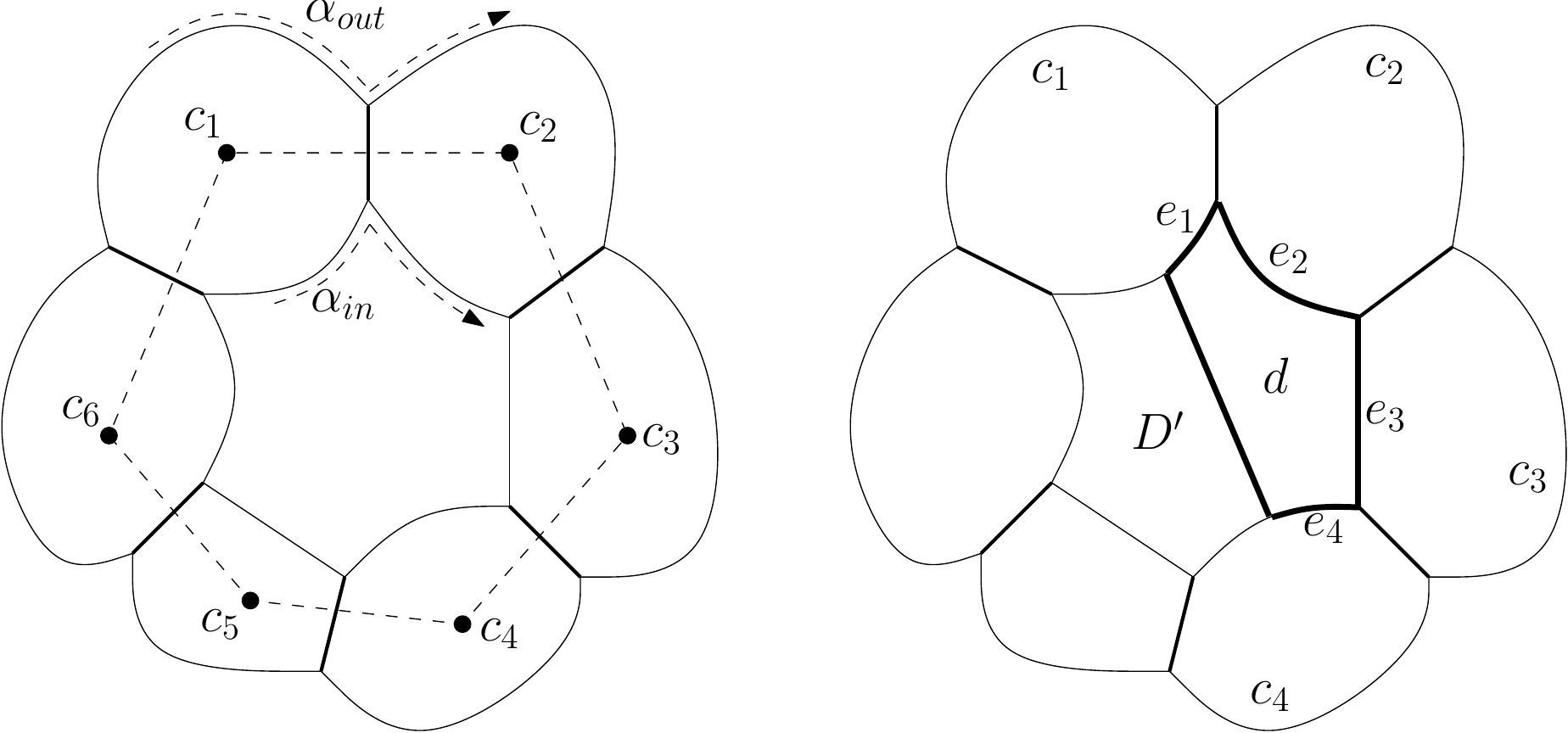}
\caption{Proof of Lemma~\ref{l:cycles} for $k=6$: the cycle $c$ and the paths $\alpha_{in}, \alpha_{out}$ (left); the exposed path in $D$ (right).}
\label{f:6cycle}
\end{figure}
Consider a soccer diagram $D$ with boundary $\ain$ that has at most $6$ pentagons --- such a diagram always 
exists by \cite[Remark 2.9]{EMcCM}. By \cite[Theorem 3.10]{EMcCM}, either $D$ contains an exposed path of length $4$, or 
$D$ is one of the two diagrams in \cite[Figure 2]{EMcCM}. It is easy to see that the latter is impossible.
Therefore, for the rest of the proof we suppose that there is an exposed path of length $4$ in $D$. 
This is impossible when $k=4$, hence further we consider $k=5,6$.
If, for some $i\in \{1,2,\ldots,k\}$, the cell $c_i$ belongs to $D$ (i.e.\ $D$ is ``outside $\ain$" in Figure~\ref{f:6cycle}) then
$c_{i-1}$ (or $c_k$ if $i=1$) and $c_{i+1}$ share an edge, contrary to our assumptions. Therefore $D$ does not contain 
any $c_i$ (i.e.\ $D$ is ``inside $\ain$" in Figure~\ref{f:6cycle}).
Without loss of generality we assume that the exposed path $p$ of length $4$ consists of edges $e_1,e_2,e_3,e_4$ such that $e_i
\subseteq c_i$. Let $d$ be a cell of $D$ containing $p$.
If $D=d$ then the lemma is proved, since then the cycle $c$ (in $Y$) is contained in the $k$--wheel $(d;c_1,\ldots,c_k)$ in $Y$.
Therefore, further we assume that $D=D' \cup d$, where $D'$ is a new (nonempty) soccer diagram. 
Assume that among the cells $c_i$ only the cells $c_1,c_2,\ldots,c_r$ intersect $d$, for $4\leqslant r \leqslant k-1$.
Then we obtain a new cycle $c'$ of cells: $c_1,d,c_r,\ldots,c_k$, of length strictly smaller than $k$. This corresponds to a full cycle $c'$ in $Y$.
If $k=5$ then this is impossible as shown above. If $k=6$ then $c'$ has to have length $5$ and, as we have just proved $c'$ is contained
in a wheel $(d';c_1,d,\ldots)$. This means that both $d$ and $d'$ are pentagons, which contradicts the definition of the soccer diagram,
and thus yields the lemma.
\end{proof}

\begin{lemma}
\label{l:7wheels}
Let $W=(v;c_1,c_2,\ldots,c_7)$ be a $7$--wheel in $X$. Then there exist vertices $y\neq z$ such that $\la v,y,z \ra \in X$, and, 
up to renaming cyclically vertices $c_i$, the vertex $y$ is adjacent to $c_1,c_2,c_3,c_4$, and $z$ is adjacent to
$c_4,c_5,c_6,c_7,c_1$. 
\end{lemma}
\begin{proof}
The first part of the proof is the same as in the proof of Lemma~\ref{l:cycles} above: We consider the ``cycle" of cells
$c_1,\ldots, c_7$ in the soccer tiling corresponding to $X_v$, and we find a path $\alpha_{in}$ of length at most
$7\cdot 4 /2=14$. Considering a soccer diagram $D$ with the boundary $\alpha_{in}$ and with at most $6$ pentagons, 
we conclude that $D$ does not
contain any $c_i$, and that there is an exposed 
path of length $4$. It is not possible that $D$ consists of a single cell, so that the cell $d$ containing the exposed path forms
a new cycle of less than $7$ cells. For the new cycle we apply Lemma~\ref{l:cycles}, to obtain a single ``filling" cell $f$.
The pair of cells $(d,f)$ corresponds to the pair of required vertices: $(y,z)$ or $(z,y)$.
\end{proof}

\medskip

\noindent
\emph{Proof of Theorem B. } 
Assume that $dW=(v_0,v_0',w;v_1,\ldots, v_{k-2};v_1',\ldots,v_{l-2}')$ is a $(k,l)$--dwheel in $X$ with the boundary length at most $8$. 
Without loss of
generality we assume that $k\geqslant l$. By Lemma~\ref{l:wheels}
the only possibilities for the pair $(k,l)$ are: $(5,5),(6,5),(6,6),(7,5)$. By Lemma~\ref{l:wheels}, there is a vertex $\ovv' \neq v_0'$
such that $(v_0';v_1',\ldots, v_{l-2}',w,v_0) \subseteq X_{\ovv'}$.

First we consider the case when $k\leqslant 6$. By Lemma~\ref{l:wheels}, there is a vertex $\ovv \neq v_0$
such that the wheel $(v_0;v_1,\ldots, v_{k-2},w,v_0') \subseteq X_{\ovv}$.
If $\ovv = \ovv'$ then $dW \subseteq B_1(\ovv)$ and we are done. Hence further we assume that $\ovv \neq \ovv'$. 
If $v_1=v_1'$ then $(v_1,\ovv,w,\ovv')$ is a $4$--cycle in $X_{\langle v_0,v_0' \rangle}$. By Lemma~\ref{l:wheels}, this cycle is not full and thus $\la \ovv,\ovv' \ra \in
X$. Hence, the vertices $v_0,v_0',v_1,\ovv,\ovv'$ span a $4$--simplex --- contradiction.
If $v_1 \neq v_1'$ then, by the definition of a dwheel $dW$, we have $v_1\sim v_1'$ and, by our assumptions, $l=5$.
In this case the cycle $(v_1,\ovv,w,\ovv',v_1')$ is a $5$--cycle in $X_{\langle v_0,v_0' \rangle}$. If this cycle is not full then we
have a contradiction as before. If it is full then the edge $\langle v_0,v_0' \rangle$ has degree $5$. However, also the edge
${\langle v_0',\ovv' \rangle}$ has degree $5$ contradicting the definition of the $5/6^{\ast}$--triangulation.

It remains to treat the case $(k,l)=(7,5)$. Observe that in this case we have $v_1=v_1'$. 
Let $y\neq z$ be vertices as in Lemma~\ref{l:7wheels} for the $7$--wheel 
$(v_0;v_1,\ldots, v_{k-2},w,v_0')$. Either $v_1,v_0',w$ are all adjacent to $y$ or to $z$, or, without loss of generality, we assume
that $\la z,w,v_0' \ra \in X$, and $ \la y,v_1,v_0' \ra \in X$. In the former case we proceed as above for $k\leqslant 6$.
Thus further we assume that $\la y,w \ra, \la z,v_1 \ra \notin X$. Then there is a $5$--cycle $(y,z,w,\ovv',v_1)$ in the link of
the edge $\la v_0,v_0' \ra$.
If this cycle is not full then we obtain a simplex of dimension above $3$ in $X$ --- contradiction. 
Therefore the cycle is full and hence the degree of  $\la v_0,v_0' \ra$ is $5$. This means however that there are two degree $5$ edges
in the triangle $\la v_0,v_0',\ovv' \ra$, contradicting the definition of the $5/6^{\ast}$--triangulation.
\hfill $\square$

\section{Further applications and final remarks}
\label{s:fin}

\subsection{Versions of $m$--location}
\label{s:new8loc}
There is another, ``more natural", version of $m$--location, that originated in fact our studies:
A flag (now, not necessarily $5$--large) simplicial complex is \emph{$m$--located} if every homotopically trivial
loop of length at most $m$ has a filling diagram with at most one internal vertex (see eg.\ \cites{JS,O-cnpc} for basics on filling diagrams). The new property does not imply the one considered in this paper, neither vice versa. 
The following analogue of Theorem A holds.
\begin{theorem*}
Let $X$ be a simply connected $8$--located flag simplicial complex. Then the $1$--skeleton of $X$, equipped with
the standard path metric, is Gromov hyperbolic.
\end{theorem*}
A more general scheme for local conditions implying Gromov hyperbolicity relies on the requirement that ``short loops allow 
small filling diagrams". This leads to particular instances of the local-to-global principle for the linear Dehn function. 

In the next subsections one may usually consider both versions of ``$m$--location" when it is mentioned.

\subsection{Further applications}
\label{s:furt}
The notion of $8$--location can be applied not only in the case of (closed) manifolds or manifolds with boundaries.
Having, a  $5/6^{\ast}$--triangulation of a $3$--manifold one can construct a branched cover that will still be an
$8$--located locally $5$--large complex. There are also various amalgamation procedures that can be applied to
obtain ``singular" spaces beginning with manifolds. Further, it is possible to construct $8$--located pseudomanifolds 
using constructions from \cite{JS} and similar (cf.\ e.g.\ \cite{O-cnpc} for further references). Objects obtained in this way
may be asymptotically different from the existing ones --- compare the next subsection.

Clearly, $8$--location applies to higher dimensional manifolds as well. Nevertheless, 
it is an interesting question, whether for $4$--manifolds (and in higher dimensions) one may formulate conditions on
a triangulation in the spirit of Thurston's condition, implying $8$--location, and thus hyperbolicity.

\subsection{Relations to other combinatorial nonpositive curvature conditions}
\label{s:other}

There are various notions of ``combinatorial nonpositive curvature" appearing in the literature --- see e.g.\ \cites{BCCGO,O-cnpc} for
more details on this. In particular, for $k\geqslant 6$, a \emph{$k$--systolic complex} \cite{JS} is a simply connected
locally $k$--large simplicial complex. In \cite{O-cnpc} the notion of a \emph{weakly systolic} complex, being a generalization
of $k$--systolicity, has been introduced. 
The following immediate corollary of Theorem A provides a local condition on a weakly systolic complex implying
its hyperbolicity. The new condition is more general than the well studied conditions for hyperbolicity: 
local $7$--largeness for $k$--systolic complexes
(see \cite[Section 2]{JS}), and $SD_2^{\ast}(7)$ for weakly systolic complexes (see \cite[Section 7]{O-cnpc}).

\begin{corollary}
\label{c:67syst}
The following condition implies hyperbolicity of a weakly systolic complex $X$: For every vertex $v$ such that 
$X_v$ is not $6$--large (respectively, is $6$--large but not $7$--large), every its neighbor $w$ has an $8$--large (respectively, $7$--large) link $X_w$.
\end{corollary}

Note however, that (weakly) systolic complexes do not fit well the world of manifolds of dimension above $2$ --- e.g.\ 
there are no $k$--systolic triangulations of such manifolds \cite{JS}. More strikingly, groups acting geometrically on 
systolic complexes are far from acting in the same way on aspherical manifolds --- see e.g.\ \cite{O-ciscg}.
Thus the $8$--location is a (asymptotically) far going generalization of $7$--systolicity, that allows still to obtain 
strong results on the structure of the complex --- as e.g.\ in the description of combinatorial balls provided
within the proof of Theorem~\ref{t:logl}.

\subsection{Beyond $8$--location}
\label{s:beyond}
In the current paper we focus on $8$--located complexes. Nevertheless, the $7$--location seems to be an interesting
nonpositive-curvature-like property, as argued below. It is not so with $6$--location --- there exist $6$--located
locally $5$--large triangulations of the $2$--sphere. 

Note that all $6$--systolic complexes are $7$--located. In particular, the $7$--location does not imply hyperbolicity:
The tiling of the Euclidean plane by equilateral triangles is $7$--located.
However, with much more effort than for the $8$--location, one can prove a local-to-global result
similar to Theorem~\ref{t:logl} also for $7$--located complexes.
We believe that there are other nonpositive-curvature-like properties satisfied by such complexes (cf.\ e.g.\ \cites{JS,O-cnpc}).
This, together with the fact that $7$--location goes far beyond systolicity, make us believe that this notion deserves further studies.


\end{document}